\documentclass[smallextended,referee,envcountsect,]{svjour3}
\smartqed
\usepackage{graphicx}
\journalname{}

\usepackage{amssymb,amsmath}
\usepackage{mathabx}
\usepackage{graphicx,multirow,amsmath,amsfonts,mathrsfs,amssymb}

\usepackage{subfigure}

 \newtheorem{assumption}{Assumption}

\usepackage{bm,color,xcolor}
\usepackage{color}
\usepackage{verbatim}

\usepackage[ruled,linesnumbered]{algorithm2e}

 \allowdisplaybreaks [4]

\begin{document}
\title{Distributionally robust chance constrained Markov decision process with Kullback-Leibler divergence 
}

\date{}

\author{Tian Xia, Jia Liu, Abdel Lisser }

\institute{Tian Xia \at
             School of Mathematics and Statistics, Xi'an Jiaotong University,
              Xi'an, 710049, P. R. China\\
            xt990221@stu.xjtu.edu.cn
           \and
              Jia Liu,  Corresponding author  \at
              School of Mathematics and Statistics, Xi'an Jiaotong University,
              Xi'an, 710049, P. R. China\\
              jialiu@xjtu.edu.cn
           \and
              Abdel Lisser \at
              Laboratoire des Signaux et des Systemes,  CentraleSupelec, 
              Gif-sur-Yvette, 91190, France\\
              abdel.lisser@centralesupelec.fr
}

\maketitle

\begin{abstract}
This paper considers the distributionally robust chance constrained Markov decision process with random reward and ambiguous reward distribution. We consider individual and joint chance constraint cases with Kullback-Leibler divergence based ambiguity sets centered at elliptical distributions or elliptical mixture distributions, respectively. We derive tractable reformulations of the distributionally robust individual chance constrained Markov decision process problems and design a new hybrid algorithm based on the sequential convex approximation and line search method for the joint case. We carry out numerical tests with a machine replacement problem.

\end{abstract}
\keywords{Markov decision process \and Chance constraint \and Distributionally robust optimization \and Kullback-Leibler divergence \and Elliptical distribution}
\subclass{90C15 \and  90C40 \and 90C99}


\section{Introduction}



Markov decision process (MDP) is an effective mathematical model to find an optimal dynamic policy in a long-term uncertain environment, which is the  fundamental mathematical tool of reinforcement learning \cite{sutton1999between}.
It has many important applications in 
healthcare \cite{goyal2022robust}, autonomous driving \cite{wei2011point}, financial markets \cite{chakraborty2019capturing}, inventory control \cite{klabjan2013robust}, game theory \cite{yu2022zero} and so on. 






The randomness of MDP often comes from two perspectives: reward and transition probabilities. Risk attitude is an important issue when the decision-maker measures the randomness of the reward. Many risk criteria have been considered in risk-aversion MDP, for instance, mean and variance \cite{xia2020risk}, semi-variance \cite{yu2022zero}, Value-at-Risk \cite{ma2019state}, Conditional Value-at-Risk \cite{prashanth2014policy}. Depending on the randomness of transition probabilities, MDP problems can be classified into two groups: rectangular MDP \cite{ramani2022robust,satia1973markovian,wiesemann2013robust} and nonrectangular MDP \cite{mannor2016robust,wang2022reliable}.








In many real applications of MDP, for instance, 
autonomous driving or healthcare, the safety requirements play an important role when making a dynamic decision to avoid extreme behaviour out of control \cite{kiran2021deep}, 
which take into account robust  constraints in the MDP problem, for instance the constrained MDP (CMDP) \cite{varagapriya2022constrained}. 
To address the extreme conservation of the robust constraints, 
chance constraints 
control the extreme loss 
in 
a probability, which has been widely applied in shape optimization, game theory, electricity market and many other fields \cite{dvorkin2019chance,jiang2022data,kuccukyavuz2022chance,liu2022distributionally,peng2021games}. 
Delage and Mannor \cite{delage2010percentile} study a reformulation of 
chance constrained MDP (CCMDP) with random rewards or transition probabilities.
Varagapriya et al. \cite{varagapriya2022joint} apply joint chance constraints into constrained MDP  and find reformulations when the rewards follow an elliptical distribution.












In some applications of CCMDP, 
the distribution of 
random parameters is 
not perfectly known, 
due to the estimation error or imperfect a-priori knowledge.
To address 
this problem, we can employ the distributionally robust optimization (DRO) approach \cite{hanasusanto2015distributionally}, where the decision-maker makes a robust decision with respect to the worst-case distribution in a pre-set ambiguity set. In
DRO literature, there are two 
major 
types of ambiguity sets: the moments-based and the distance-based. In moments-based DRO \cite{delage2010distributionally,wiesemann2014distributionally}, decision-maker knows some moments information about 
of random parameters. In distance-based DRO, the decision-maker 
has a reference distribution and consider a ball centered at it with respect to a probability distance,
given that she/he 
believes that the true distribution of random parameters is close to the reference distribution.
Depending on the probability distance we choose, there are $\phi$-divergence (including Kullback-Leibler (K-L) divergence as an important case) distance based DRO \cite{hu2013kullback,jiang2016data} and Wasserstein distance based DRO \cite{chen2022data,gao2022distributionally,ji2021data,xie2021distributionally}. Applying the techniques of DRO into CCMDP, we have the 
distributionally robust chance constrained MDP (DRCCMDP) problem. 
Nguyen 
et al. \cite{nguyen2022distributionally} studied individual DRCCMDP 
with moments-based, $\phi$-divergence based and Wasserstein distance based 
ambiguity sets.
However, the study of DRCCMDP is far from completeness. There are still many important problems for research, for instance, the joint chance constraint in DRCCMDP has not been studied, the high-kurtosis, fat-tailedness or 
multimodality of the reference distribution in distance-based DRCCMDP are not considered.




In this paper, we study
the K-L divergence distance based DRCCMDP (KL-DRCCMDP) when the transition probabilities are known and the reward vector is random {,which is the same setting as in \cite{nguyen2022distributionally,varagapriya2022constrained,varagapriya2022joint}. Like \cite{nguyen2022distributionally}, we assume that we only know partial information about the distribution of random reward and apply the distributionally robust optimization approach.
Unlike \cite{delage2010percentile}, 
we consider random reward in both the 
objective function and the constraints in order to 
model more safe scenes in real life. Compared with 
\cite{nguyen2022distributionally} {where the authors consider only} the individual case ,
we study both individual and joint chance constraint cases in the MDP settings in order to characterize the overall satisfaction of safe constraints 
with K-L divergence distance based ambiguity set centered at an elliptical reference distribution.} 
We derive reformulations of the related optimization problems in these two cases. For the individual case, the reformulation is convex. However for the joint case, the reformulation is not convex. We design a new hybrid algorithm based on the sequential convex approximation and line search method to solve this nonconvex problem. In the last part of the joint case of KL-DRCCMDP, we study the case where the ambiguity set is centered at an elliptical mixture distribution and derive a new reformulation. Finally we conduct numerical experiments on a machine replacement problem to test our models and algorithms. The major contributions of this paper are listed below.
\begin{itemize}
    \item As far as we know, 
    this is the first work studying 
    the 
    joint case of DRCCMDP.
    \item We consider an
    elliptical reference distribution and an elliptical mixture reference distribution as the center of the ambiguity sets, which can reflect the high-kurtosis, fat-tailedness or 
multimodality 
of the a-priori information.
    \item We propose a new hybrid algorithm based on a sequential convex approximation and line search method to solve
    the nonconvex reformulation. 
    Numerical results
    validate the practicability of this algorithm.
\end{itemize}
In Section \ref{j2}, we introduce a series of fundamental models of MDP as the background of DRCCMDP. In Section \ref{j3}, we study three kinds of KL-divergence based DRCCMDP: the individual KL-DRCCMDP with elliptical reference distributions, the joint KL-DRCCMDP with elliptical reference distributions and the joint KL-DRCCMDP with elliptical mixture reference distributions. In Section \ref{j4}, we carry out numerical experiments on a machine replacement problem. In the last section, we give the conclusion.

\section{Background of DRCCMDP}\label{j2}
\subsection{MDP}
We consider an infinite horizon discrete time Markov decision process (MDP) problem, which can be represented as a tuple $(\mathcal{S}, \mathcal{A}, P, r_0, q, \beta),$ where:
\begin{itemize}
\item[$\bullet$]$\mathcal{S}$ is a finite state space with $|S|$ states whose generic element is denoted by $s$.
\item[$\bullet$]$\mathcal{A}$ is a finite action space with $|\mathcal{A}|$ actions and $a\in\mathcal{A}(s)$ denotes the action $a$ at state $s.$
\item[$\bullet$]$P\in\mathbb{R}^{|\mathcal{S}|\times|\mathcal{A}|\times|\mathcal{S}|}$ is the distribution of the transition probability $p(\overline{s}|s,a),$ which denotes the probability of moving from state $s$ to $\overline{s}$ when the action $a\in\mathcal{A}(s)$ is taken.
\item[$\bullet$]$r_0(s,a)_{s\in\mathcal{S}, a\in\mathcal{A}(s)}:\mathcal{S}\times\mathcal{A}\rightarrow\mathbb{R}$ denotes a running reward, which is the reward at the state $s$ when the action $a$ is taken. $r_0=(r_0(s,a))_{s\in\mathcal{S}, a\in\mathcal{A}(s)}\in\mathbb{R}^{|\mathcal{S}|\times|\mathcal{A}|}$ is the running reward vector.
\item[$\bullet$]$q=(q(s))_{s\in\mathcal{S}}$ represents the probability of the initial state.
\item[$\bullet$]$\beta\in[0,1)$ is the discount factor.
\end{itemize}
 In an MDP, the agent aims at maximizing her/his value function with respect to the whole trajectory by choosing an optimal policy. By \cite{sutton1999policy}, it is worth noting that there are two ways of formulating the agent's objective. One is the average reward formulation, the other is considering a discounting factor $\beta\in[0,1).$ As we care more about the long-term reward obtained from the MDP, we pay more attention on optimizing current rewards over future rewards. Therefore, we follow the latter formulation to consider discounting value function in this paper.

For a discrete time controlled Markov chain $(s_t,a_t)_{t=0}^{\infty}$ defined on the state space $\mathcal{S}$ and action space $\mathcal{A}$, where $s_t$ and $a_t$ are the state and action at time $t$ respectively. At the initial time $t=0,$ the state is $s_{0}\in\mathcal{S},$ and the action $a_0\in\mathcal{A}(s_0)$ is taken according to the initial state's probability $q.$ Then the agent gains rewards $r_0(s_0,a_0)$ based on the current state and action. When $t=1,$ the state moves to $s_1$ with the transition probability $p(s_1|s_0,a_0).$ The dynamics of the MDP repeat at state $s_1$ and continue in the following infinite time horizon. As a result, we are able to get the value function for the whole process.

We assume that running rewards $r$ and transition probabilities $p$ are stationary, i.e. they only depend on states and actions rather than on time. We define the policy $\pi=(\mu(a|s))_{s\in\mathcal{S},a\in\mathcal{A}(s)}\in\mathbb{R}^{|\mathcal{S}|\times|\mathcal{A}|}$ where $\mu(a|s)$ denotes the probability that the action $a$ is taken at state $s$, and $\xi_t=\{s_0,a_0,s_1,a_1,...,s_{t-1}$, $a_{t-1},s_t\}$ the whole historical trajectory at time $t.$ 
For different time $t,$ sometimes the decisions made by the agent may vary accordingly, thus the chosen policy may vary depending on time. We call this kind of policy the history dependent policy denoted as $\pi_h=(\mu_{t}(a|s))_{s\in\mathcal{S},a\in\mathcal{A}(s)}, t=1,2,...,\infty.$ When the policy is independent of time, we call it a stationary policy. That is, there exists a vector $\overline{\pi}$ such that $\pi_h=(\mu_{t}(a|s))_{s\in\mathcal{S},a\in\mathcal{A}(s)}=\overline{\pi}=(\overline{\mu}(a|s))_{s\in\mathcal{S},a\in\mathcal{A}(s)}$ for all $t.$ Let $\Pi_h$ and $\Pi_s$ be the sets of all possible history dependent policies and stationary policies respectively. When the reward $r_0(s,a)$ is random, for a fixed $\pi_h\in\Pi_h,$ we consider the discounted expected value function 
\begin{equation}\label{00}
    V_{\beta}(q,\pi_h)=\sum_{t=0}^{\infty}\beta^{t}\mathbb{E}_{q,\pi_h}(r_0(s_t,a_t)),
\end{equation}
where $\beta\in [0,1)$ is the fixed discount factor. The objective of the agent is to maximize the discounted expected value function
\begin{equation}\label{optim MDP}
    \max_{\pi_{h}\in\Pi_h}{\sum_{t=0}^{\infty}\beta^{t}\mathbb{E}_{q,\pi_h}(r_0(s_t,a_t))}.
\end{equation}

We denote by $d_{\beta}(q,\pi_h)$ the $\beta$-discounted occupation measure such that 
$$
\begin{aligned}
d_{\beta}(q,\pi_h,s,a) =(1-\beta)\sum_{t=0}^{\infty}\beta^{t}p_{q,\pi_{h}}(s_{t}=s, a_{t}=a), \forall{s\in\mathcal{S},a\in\mathcal{A}(s)}.
\end{aligned}
$$
As the state and action spaces are both finite, by Theorem 3.1 in \cite{altman1999constrained}, the occupation measure $d_{\beta}(q,\pi_h,s,a)$ is a well-defined probability distribution. By taking the occupation measure in \eqref{00}, the discounted expected value function can be written as
$$
\begin{aligned}
V_{\beta}(q,\pi_h) & =\sum_{(s,a)\in\Lambda}\sum_{t=0}^{\infty}\beta^{t}p_{q,\pi_h}(s_{t}=s,a_{t}=a) r_{0}(s,a) \\
& =\frac{1}{1-\beta}\sum_{(s,a)\in\Lambda\label{value}}d_{\beta}(q,\pi_h,s,a) r_0(s,a),
\end{aligned} 
$$ where we define $\Lambda=\left\{(s,a)|s\in\mathcal{S}, a\in\mathcal{A}(s)\right\}$.

By Theorem 3.2 in \cite{altman1999constrained}, we know that the set of occupation measures corresponding to history dependent policies is the same as that concerning stationary policies. Furthermore, from \cite{altman1999constrained,varagapriya2022constrained} we have: 
\begin{lemma}[\cite{altman1999constrained,varagapriya2022constrained}]\label{th1}
    The set of occupation measures corresponding to history dependent policies is equal to the set
\begin{equation}\label{le}
    \Delta_{\beta,q}=\left\{\tau \in \mathbb{R}^{|\mathcal{S}|\times|\mathcal{A}|}\Bigg|
    \begin{array}{l}
\sum\limits_{(s,a)\in\Lambda}\tau(s,a)\left(\delta(s',s)-\alpha p(s'|s,a)\right)=(1-\beta)q(s'), \\
\tau(s,a)\ge0, \forall s',s\in\mathcal{S},a\in\mathcal{A}(s).
\end{array}\right\},
\end{equation} where $\delta(s',s)$ is the Kronecker delta, such that the expected discounted value function defined by \eqref{value} remains invariant to time. 
\end{lemma}

Therefore the MDP problem \eqref{optim MDP} with history dependent policies can be equivalently written as a stationary MDP problem:
\begin{subequations}\label{9pp}
\begin{eqnarray}
& \max\limits_{\tau} & \frac{1}{1-\beta}\sum\limits_{(s,a)\in\Lambda}\tau(s,a)r_0(s,a)\\
&{\rm s.t.} & \tau\in\Delta_{\beta,q}.
\end{eqnarray}
\end{subequations}  If $\tau^{\star}$ is an optimal solution of \eqref{9pp}, then the stationary policy $f^{\star}$ is given by $f^{\star}(s,a)=\frac{\tau^{\star}(s,a)}{\sum\limits_{a\in\mathcal{A}(s)}\tau^{\star}(s,a)}$ for all $(s,a)\in \Lambda$, whenever the denominator is nonzero.

\subsection{Constrained MDP}
In a constrained MDP (CMDP), on the basis of the MDP defined above, we consider extra constraints on some additional rewards. Let $r_{k}(s,a)_{(s,a)\in\Lambda}:\mathcal{S}\times\mathcal{A}\rightarrow\mathbb{R}$ be the rewards for the $k-$th constraint under state $s$ and action $a$, $k=1,\dots,K$, and $K$ denotes the number of constraints. We denote $r_k=(r_{k}(s,a))_{(s,a)\in\Lambda}\in\mathbb{R}^{|\Lambda|}$ as the rewards vector 
for the $k-$th constraint. Let $\Xi=(\xi_k)_{k=1}^{K}$ be the set of lower bounds for the
constraints.
A CMDP is then defined by the tuple $(\mathcal{S}, \mathcal{A}, R, \Xi, P, q, \beta)$ where $R=(r_k)_{k=0}^{K}.$

In order to focus on optimizing current rewards rather than future ones, we apply the discount factor in the expected constrained value function, which is defined as $$\phi_{k,\beta}(q,\pi_h)=\frac{1}{1-\beta}\sum_{(s,a)\in\Lambda}d_{\beta}(q,\pi_h,s,a) r_k(s,a)$$
for the $k$-th expected constraint. Combined with Lemma \ref{th1}, we can formulate the objective of a CMDP as the following optimization problem
\begin{subequations}
\begin{eqnarray}
& \max\limits_{\tau} & \frac{1}{1-\beta}\sum\limits_{(s,a)\in\Lambda}\tau(s,a)r_0(s,a)\\
&{\rm s.t.} & \sum\limits_{(s,a)\in\Lambda}\tau(s,a)r_k(s,a)\ge \xi_{k},  k=1,2,...,K\\
&&\tau\in\Delta_{\beta,q}.
\end{eqnarray}
\end{subequations}



\subsection{Chance constrained MDP}

In many applications, the reward vectors $r_{k}, k=0,1,...,K$ are random. It is reasonable to consider the MDP with random reward. In this vein, we can use chance constraints to ensure the constraints in the CMDP hold with a large probability. We denote it as the chance-constrained MDP (CCMDP).

For the $k$-th random constrained rewards vector $r_k=(r_{k}(s,a))_{(s,a)\in\Lambda}$, we assume its probability distribution is $F_k$, $k=0,1,...,K$. We preset a confidence vector $\epsilon=(\epsilon_k)_{k=1}^{K}$ for the CCMDP, where $\epsilon_k\in[0,1]$. Then we can define the individual CCMDP (I-CCMDP) as a tuple $(\mathcal{S}, \mathcal{A}, R, \Xi, P, \mathcal{D}, q, \beta, \epsilon)$, where $\mathcal{D}=(F_k)_{k=0}^{K}$, which can be reformulated as the following optimization problem:
\begin{subequations}
\begin{eqnarray}
\rm{(I-CCMDP)} & \max\limits_{\tau} & \frac{1}{1-\beta}\mathbb{E}_{F_0}[\tau^{\top}\cdot r_0]\\
&{\rm s.t.} & \mathbb{P}_{F_k}(\tau^{\top}\cdot r_k\ge \xi_{k})\ge\epsilon_k,  k=1,2,...,K\\
&&\tau\in\Delta_{\beta,q}.
\end{eqnarray}
\end{subequations}

The joint CCMDP (J-CCMDP) can be defined as a tuple $(\mathcal{S}, \mathcal{A}, R, \Xi, P, F$, $q, \beta, \hat{\epsilon})$, where $\hat{\epsilon}\in[0,1]$. The J-CCMDP can be reformulated as 
\begin{subequations}
\begin{eqnarray}
\rm{(J-CCMDP)} & \max\limits_{\tau} & \frac{1}{1-\beta}\mathbb{E}_{F_0}[\tau^{\top}\cdot r_{0}]\\
&{\rm s.t.} & \mathbb{P}_{F}(\tau^{\top}\cdot r_k\ge \xi_{k},  k=1,2,...,K)\ge\hat{\epsilon}, \\
&&\tau\in\Delta_{\beta,q},
\end{eqnarray}
\end{subequations} here $F$ denotes the joint probability distribution of $r_1, r_2,...,r_K$ when $\hat{\epsilon}$ denotes the overall confidence for $K$ constraints.


\subsection{Distributionally robust chance constrained MDP}
Based on the CCMDP defined above, if the information of distributions of rewards $r_k$ are not perfectly known, we can apply the distributionally robust optimization approach to handle the uncertainty of $\hat{F}$ or $F_{k},k=0,...,K$. Then we consider a distributionally robust chance constrained MDP (DRCCMDP).

The individual DRCCMDP (I-DRCCMDP) can be defined as the tuple $(\mathcal{S}, \mathcal{A}, R$, $\Xi, P, \mathcal{D}, \tilde{\mathcal{F}}, q, \beta, \epsilon)$, where $\tilde{\mathcal{F}}=(\mathcal{F}_k)_{k=0}^{K}$ when $\mathcal{F}_k$ denotes the ambiguity set for the random distribution $F_k$. Therefore the I-DRCCMDP can be reformulated as the following optimization problem:
\begin{subequations}\label{obj MDP}
\begin{eqnarray}
\rm{(I-DRCCMDP)}  & \max\limits_{\tau} & \inf\limits_{F_{0}\in\mathcal{F}_0} \   \frac{1}{1-\beta}\mathbb{E}_{F_0}[\tau^{\top}\cdot r_0] \label{kj} \\
& {\rm s.t.} & \inf\limits_{F_{k}\in\mathcal{F}_k} \ \mathbb{P}_{F_k}\label{pd}(\tau^{\top}\cdot r_k\ge \xi_{k})\ge\epsilon_k,\   k=1,2,...,K, \\
&&\tau\in\Delta_{\beta,q}.
\end{eqnarray}
\end{subequations}

The joint DRCCMDP (J-DRCCMDP) can be defined as the tuple $(\mathcal{S}, \mathcal{A}, R$, $\Xi, P, \mathcal{D}, \mathcal{F}_0, \mathcal{F}, q, \beta, \hat{\epsilon})$, where $\hat{\epsilon}\in[0,1]$, $\mathcal{F}_0$ denotes the ambiguity set for the unknown
distribution $F_0$ and $\mathcal{F}$ denotes the ambiguity set for the unknown 
joint distribution $F$ of $r_{1}, r_2,...,r_k$. The J-DRCCMDP can be reformulated as
\begin{subequations}\label{Jobj}
\begin{eqnarray}
\rm{(J-DRCCMDP)}  & \max\limits_{\tau}&\inf\limits_{F_{0}\in\mathcal{F}_0} \  \frac{1}{1-\beta}\mathbb{E}_{F_0}[\tau^{\top}\cdot r_0]\label{12a}\\
 & {\rm s.t.} & \inf\limits_{F\in\mathcal{F}}\  \label{JF}\mathbb{P}_{F}(\tau^{\top}\cdot r_k\ge \xi_{k}, k=1,2,...,K)\ge\hat{\epsilon}, \\
&&\tau\in\Delta_{\beta,q}.
\end{eqnarray}
\end{subequations}

\section{K-L divergence based DRCCMDP}\label{j3}

In this Section, we study DRCCMDP with the K-L divergence distance \cite{joyce2011kullback} based ambiguity set (K-L DRCCMDP). In Section \ref{se3.1}, we study the individually K-L DRCCMDP (K-L I-DRCCMDP) with an elliptical reference distribution. In Section \ref{se3.2}, we study the jointly K-L DRCCMDP (K-L J-DRCCMDP) with an elliptical reference distribution. In Section \ref{se3.3}, we study the K-L J-DRCCMDP with an elliptical mixture reference distribution.

\begin{definition}
Let $D_{\rm{KL}}$ denotes the Kullback-Leibler divergence distance
$$ D_{\rm{KL}}({F}_{k}||\tilde{F}_{k})=\int_{\Omega_{k}} \phi\left(\frac{f_{{F}_{k}}(r_{k})}{f_{\tilde{F}_{k}}(r_{k})}\right)f_{\tilde{F}_{k}}(r_{k}) dr_{k},$$
where $\tilde{F}_{k}$ is the reference distribution of $r_{k}$, $f_{{F}_{k}}(r_{k})$ and $f_{\tilde{F}_{k}}(r_{k})$ are the density functions of the true distribution and the reference distribution of $r_{k}$ on support $\Omega_{k}$ respectively. $\phi(t)$ is defined as follows 
$$
\phi(t)=\left\{\begin{array}{ll}
t{\rm{log}}t-t+1, & t\ge0,\\
\infty, & t<0.
\end{array}\right.
$$
\end{definition}

\subsection{K-L I-DRCCMDP with elliptical reference distribution}\label{se3.1}
In this Section, we study the I-DRCCMDP whose ambiguity sets are based on the K-L divergence distance.

\begin{assumption}\label{12}
The ambiguity sets are $$\mathcal{F}_{k}=\left\{{F}_{k}|D_{\rm{KL}}({F}_{k}||\tilde{F}_{k})\le \delta_{k}\right\}, k=0,1,...,K,$$
where $\tilde{F}_{k}$ is the reference distribution of reward vector $r_{k}$, the radius $\delta_{k}$ controls the size of the ambiguity sets. 
\end{assumption}

We assume that the reference distribution belongs to the elliptical distribution family. 

\begin{definition}[\cite{fang2018symmetric}]\label{906}
A d-dimensional vector $X$ 
follows an elliptical distribution  $E_{d}(\mu, \Sigma, \psi)$ if its characteristic function has the form $\mathbb{E}(e^{ib^{\top} X})=e^{ib^{\top} \mu}\psi(b^{\top}\Sigma b),$ where $\mu\in\mathbb{R}^d$ is the location parameter, $\Sigma\in\mathbb{R}^{d\times d}$ is the dispersion matrix, $\psi$ is the characteristic generator.
\end{definition}

The elliptical distribution has the following property.
\begin{lemma}[\cite{mcneil2015quantitative}]\label{2}
Let $X\sim E_{d}(\mu, \Sigma, \psi)$. For any $B\in\mathbb{R}^{k\times d}$ and $b\in\mathbb{R}^{d}$, $$BX+b\sim E_{d}(B\mu+b, B\Sigma B^{\top}, \psi).$$
As a special case, if $a\in\mathbb{R}^{d}$, then $$a^{\top}X\sim E_{1}(a^{\top}\mu, a^{\top}\Sigma a, \psi).$$
\end{lemma}

A random vector $X$ follows a multivariate log-elliptical distribution with parameters $\mu$ and $\Sigma$ if $\log{X}$ follows an elliptical distribution: $$\log{X}\sim E_{d}(\mu, \Sigma, \psi),$$ which can be denoted as $X\sim LE_{d}(\mu, \Sigma, \psi)$. The following lemma defines the expectation of log-elliptical distributions.
\begin{lemma}[\cite{hamada2008capm}]\label{3}
Let $X\sim LE_{d}(\mu, \Sigma, \psi)$ with components $X_1, X_2,...,X_K$. If the mean of $X_{k}$ exists, then $$\mathbb{E}(X_{k})=e^{\mu_{k}}\psi(-\frac{1}{2}\sigma_{k}^2),$$ where $\mu_{k}$ and $\sigma_{k}^2$ are the mean and variance of $X_{k}$ respectively.
\end{lemma}

Gaussian, Laplace and Generalized stable laws distributions are all elliptical distributions with different characteristic generator $\psi:[0,+\infty)\rightarrow\mathbb{R}$ as shown in Table 1.
\begin{table}[h]
\centering
\caption{\normalsize The characteristic generator of three different elliptical distributions}
\begin{tabular}{c|c|c|c}
\hline
Distribution & Gaussian & Laplace & Generalized stable laws \\ \hline
Characteristic generator $\psi(t)$ & $e^{-t}$ & $\frac{1}{1+t}$ & $e^{-\omega_{1}t^{\frac{\omega_2}{2}}}$, $\omega_{1},\omega_{2}>0$ \\ \hline
\end{tabular}
\end{table}

Before considering the reformulation of K-L I-DRCCMDP centered at elliptical distributions, we cite two important lemmas from literature.

\begin{lemma}[\cite{hu2013kullback}]\label{98}
Given Assumption \ref{12}, the objective function in \eqref{kj} is equivalent to $$\inf\limits_{\alpha\in[0,+\infty)} \alpha {\rm{log}} \mathbb{E}_{\tilde{F}_{0}}\left[\exp(-\frac{\tau^{\top}r_{0}}{\alpha})\right]+\alpha \delta_{0}.$$
\end{lemma}

\begin{lemma}[\cite{jiang2016data}]\label{57}
Given Assumption \ref{12}, the constraint \eqref{pd} is equivalent to $$\mathbb{P}_{\tilde{F}_{k}}(\tau^{\top}r_{k}\ge\xi_{k})\ge \tilde{\epsilon}_k, k=1,2,...,K,$$
where $\tilde{\epsilon}_k=\inf\limits_{x\in(0,1)}\{\frac{e^{-\delta_{k}}x^{\epsilon_{k}}-1}{x-1}\}$.
\end{lemma}

Based on these two Lemmas, we get the following reformulation of \eqref{obj MDP}.
\begin{theorem}\label{1a}
We study the ambiguity set in Assumption \ref{12}. We assume the reference distribution $\tilde{F}_{k}\sim E_{|\Lambda|}(\mu_{k}, \Sigma_{k}, \psi_{k}), k=0,1,...,K$, $\Sigma_0$ is a positive definite matrix, $\psi_{0}$ is continuous and 
$\inf\limits_{t\le 0}\psi_{0}(t)\ge e^{-\delta_{0}}$. Then (I-DRCCMDP) problem \eqref{obj MDP} is equivalent to
\begin{subequations}\label{ujn}
\begin{eqnarray}
& \min\limits_{\tau,\alpha} & -\tau^{\top}\mu_{0}+\alpha\log{[\psi_{0}(-\frac{\tau^{\top}\Sigma_{0}\tau}{2\alpha^2})]}+\alpha\delta_{0},\\
& {\rm{s.t.}} & \tau^{\top}\mu_{k}+\Phi_{k}^{-1}(1-\tilde{\epsilon}_{k})\sqrt{\tau^{\top}\Sigma_{k}\tau}\ge\xi_{k}, k=1,2,\dots,K,\\
&& \alpha\ge 0,\\
&&\tau\in \Delta_{\beta,q},
\end{eqnarray}
\end{subequations} where $\Phi_{k}$ is the CDF of the variable $Z_{k}\sim E_{1}(0,1,\psi_{k})$, $\tilde{\epsilon}_k=\inf\limits_{x\in(0,1)}\{\frac{e^{-\delta_{k}}x^{\epsilon_{k}}-1}{x-1}\}$.
\end{theorem}
\begin{proof}  By Lemma \ref{98} and \ref{57}, problem \eqref{obj MDP} is equivalent to
\begin{subequations}
\begin{eqnarray}
&\min\limits_{\tau} & \inf\limits_{\alpha\in[0,+\infty)}  \alpha {\rm{log}} \mathbb{E}_{F_{0}}\left[\exp(-\frac{\tau^{\top}r_{0}}{\alpha})\right]+\alpha \delta_{0},\label{5}\\
& {\rm{s.t.}} & \mathbb{P}_{\tilde{F}_{k}}(\tau^{\top}r_{k}\ge\xi_{k})\ge \tilde{\epsilon}_k, k=1,2,...,K,\label{6}\\
&& \tau\in \Delta_{\beta,q},
\end{eqnarray}
\end{subequations} where $\tilde{\epsilon}_{k}$ is defined in Lemma \ref{57}.

As $r_0$ follows an elliptical distribution $E_{|\Lambda|}(\mu_{0}, \Sigma_{0}, \psi_{0})$, we have by Lemma \ref{2} that $-\frac{\tau^{\top}r_{0}}{\alpha}$ still follows an elliptical distribution with mean value $-\frac{\tau^{\top}\mu_{0}}{\alpha}$, variance $\frac{\tau^{\top}\Sigma_{0}\tau}{\alpha^2}$ and characteristic generator $\psi_0$. By Lemma \ref{3}, $\exp(-\frac{\tau^{\top}r_{0}}{\alpha})$ follows a log-elliptical distribution with mean value $e^{-\frac{\tau^{\top}\mu_{0}}{\alpha}}\psi_{0}(-\frac{\tau^{\top}\Sigma_{0}\tau}{2\alpha^2})$. Therefore \eqref{5} is equivalent to
\begin{equation}\label{16}
\min\limits_{\tau}\inf\limits_{\alpha\in[0,+\infty)} -\tau^{\top}\mu_{0}+\alpha\log{[\psi_{0}(-\frac{\tau^{\top}\Sigma_{0}\tau}{2\alpha^2})]}+\alpha\delta_{0}.
\end{equation} By the assumption that $\Sigma_0$ is positive definite, we have $-\frac{\tau^{\top}\Sigma_{0}\tau}{2\alpha^2}\le 0$. Then as $\inf\limits_{t\le 0}\psi_{0}(t)\ge e^{-\delta_{0}}$, we have 
$$
\begin{aligned}
&\inf\limits_{\alpha\in[0,+\infty)} -\tau^{\top}\mu_{0}+\alpha\log{[\psi_{0}(-\frac{\tau^{\top}\Sigma_{0}\tau}{2\alpha^2})]}+\alpha\delta_{0} &\\
\geqslant & \inf\limits_{\alpha\in[0,+\infty)}-\tau^{\top}\mu_{0}+\alpha\left[\log\left(\inf\limits_{t\le 0}\psi_{0}(t)\right)+\delta_{0}\right] \\
 \geqslant & \inf\limits_{\alpha\in[0,+\infty)}-\tau^{\top}\mu_{0}+\alpha\left[\log\left(e^{-\delta_0}\right)+\delta_{0}\right]=-\tau^{\top}\mu_{0}.
\end{aligned}
$$ Since $\psi_{0}$ is continuous w.r.t. $\alpha$ when $\alpha\ge 0$, the inner function of \eqref{16} is continuous w.r.t. $\alpha$. Also we have $\inf\limits_{\alpha\in[0,+\infty)} -\tau^{\top}\mu_{0}+\alpha\log{[\psi_{0}(-\frac{\tau^{\top}\Sigma_{0}\tau}{2\alpha^2})]}+\alpha\delta_{0}=\inf\limits_{\alpha\in[0,+\infty]} -\tau^{\top}\mu_{0}+\alpha\log{[\psi_{0}(-\frac{\tau^{\top}\Sigma_{0}\tau}{2\alpha^2})]}+\alpha\delta_{0}$. Therefore by Weierstrass Theorem, there exists $\alpha^{*}\in[0,+\infty]$ such that when $\alpha=\alpha^{*}$, the inner infimum term of \eqref{16} reaches its optimal value. Therefore, \eqref{16} is equivalent to
\begin{subequations}
\begin{eqnarray}
& \min\limits_{\tau,\alpha} & -\tau^{\top}\mu_{0}+\alpha\log{[\psi_{0}(-\frac{\tau^{\top}\Sigma_{0}\tau}{2\alpha^2})]}+\alpha\delta_{0},\\
& {\rm{s.t.}} & \alpha\ge 0.
\end{eqnarray}
\end{subequations} Moreover, \eqref{6} is equivalent to $\mathbb{P}_{\tilde{F}_{k}}(\frac{\tau^{\top}r_{k}-\tau^{\top}\mu_{k}}{\sqrt{\tau^{\top}\Sigma_{k}\tau}}\ge\frac{\xi_{k}-\tau^{\top}\mu_{k}}{\sqrt{\tau^{\top}\Sigma_{k}\tau}})\ge \tilde{\epsilon}_k$, $k=1,2,...,K$. Let $Z_{k}=\frac{\tau^{\top}r_{k}-\tau^{\top}\mu_{k}}{\sqrt{\tau^{\top}\Sigma_{k}\tau}}$. By Lemma \ref{2}, we know that $Z_{k}\sim E_{1}(0,1,\psi_{k})$. We denote $\Phi_{k}(z)=\mathbb{P}(Z_{k}\le z)$ as the CDF of $Z_{k}$. Then \eqref{6} is equivalent to $\frac{\xi_{k}-\tau^{\top}\mu_{k}}{\sqrt{\tau^{\top}\Sigma_{k}\tau}}\le\Phi_{k}^{-1}(1-\tilde{\epsilon}_{k})$, which can be written as $$\tau^{\top}\mu_{k}+\Phi_{k}^{-1}(1-\tilde{\epsilon}_{k})\sqrt{\tau^{\top}\Sigma_{k}\tau}\ge\xi_{k}, k=1,2,\dots,K. $$ 
\end{proof}
\qed

\subsection{K-L J-DRCCMDP with elliptical reference distribution}\label{se3.2}
In this section, we assume that the ambiguity sets in different rows are jointly independent. 
\begin{assumption}\label{ji9}
The joint K-L ambiguity set with jointly independent rows is $$\mathcal{F}:=\mathcal{F}_{1}\times\cdots\times\mathcal{F}_{K}=\left\{F=F_{1}\times\cdots\times F_{K}|F_{k}\in\mathcal{F}_{k}, k=1,...,K \right\},$$
where $F$ is the joint distribution of $r_{1},r_{2},...,r_{K}$ with jointly independent marginals $F_{1},...,F_{K}$, and $\mathcal{F}_{k}$ is a K-L ambiguity set with reference marginal distribution $\tilde{F}_{k}$ and radius $\delta_k, k=1,...,K$. 
\end{assumption}

\begin{theorem}\label{63}
Consider $\mathcal{F}_{0}$ defined in Assumption \ref{12} and $\mathcal{F}:=\mathcal{F}_{1}\times\cdots\times\mathcal{F}_{K}$ 
defined in Assumption \ref{ji9}. Assume the reference marginal distribution $\tilde{F}_{k}\sim E_{|\Lambda|}(\mu_{k}, \Sigma_{k},\psi_{k})$, $k=0,1,\dots,K$, $\Sigma_0$ is positive definite, $\psi_{0}$ is continuous and 
$\inf\limits_{t\le 0}\psi_{0}(t)\ge e^{-\delta_{0}}$. Then the (J-DRCCMDP) problem \eqref{Jobj} is equivalent to
\begin{subequations}
\begin{eqnarray}
& \min\limits_{\tau,\alpha,y} & -\tau^{\top}\mu_{0}+\alpha\log{[\psi_{0}(-\frac{\tau^{\top}\Sigma_{0}\tau}{2\alpha^2})]}+\alpha\delta_{0},\label{3999}\\
& {\rm{s.t.}} & \tau^{\top}\mu_{k}+\Phi_{k}^{-1}(1-\tilde{y}_{k})\sqrt{\tau^{\top}\Sigma_{k}\tau}\ge\xi_{k}, k=1,2,\dots,K,\\
&& 0\le y_{k}\le 1, k=1,2,\dots,K,\\
&& \prod_{k=1}^{K}y_{k}\ge\hat{\epsilon},\\
&& \alpha\ge 0,\\
&& \tau\in\Delta_{\beta,q},
\end{eqnarray}
\end{subequations} where $\tilde{y}_{k}=\inf\limits_{x\in(0,1)}\{\frac{e^{-\delta_{k}}x^{y_{k}}-1}{x-1}\}$.
\end{theorem}
\begin{proof}  As $\mathcal{F}_{0}$ is defined in Assumption \ref{12} and \eqref{12a} is the same to \eqref{kj}, with the same assumption of $\Sigma_{0}$ and $\psi_0$, we have \eqref{12a} is equivalent to
\begin{subequations}\label{39}
\begin{eqnarray}
& \min\limits_{\tau,\alpha} & -\tau^{\top}\mu_{0}+\alpha\log{[\psi_{0}(-\frac{\tau^{\top}\Sigma_{0}\tau}{2\alpha^2})]}+\alpha\delta_{0},\\
& {\rm{s.t.}} & \alpha\ge 0.
\end{eqnarray}
\end{subequations} As $F_{k}$ are pairwise independent, constraint \eqref{JF} is equivalent to
\begin{equation}\label{40}
\prod_{k=1}^{K}\inf\limits_{{F}_{k}\in\mathcal{F}_{k}}\mathbb{P}_{{F}_{k}}(\tau^{\top}\cdot r_k\ge \xi_{k})\ge\hat{\epsilon}.
\end{equation} By Lemma \ref{57} and  introducing auxiliary variables $y_{k}\in\mathbb{R}_{+}$, \eqref{40} is equivalent to
\begin{equation}\label{41}
\mathbb{P}_{\tilde{F}_{k}}(\tau^{\top}\cdot r_k\ge \xi_{k})\ge \tilde{y}_{k}, k=1,2,\dots,K,
\end{equation} 
\begin{equation}\label{42}
\prod_{k=1}^{K}y_{k}\ge\hat{\epsilon}, 0\le y_{k}\le 1, k=1,2,\dots,K,
\end{equation} where $\tilde{y}_{k}=\inf\limits_{x\in(0,1)}\{\frac{e^{-\delta_{k}}x^{y_{k}}-1}{x-1}\}$.
As the reference distribution $\tilde{F}_k$ 
is an elliptical 
 distribution, 
 following the similar reformulation procedure for \eqref{6} in Theorem \ref{1a}, we have 
\begin{equation}\label{43}
\tau^{\top}\mu_{k}+\Phi_{k}^{-1}(1-\tilde{y}_{k})\sqrt{\tau^{\top}\Sigma_{k}\tau}\ge\xi_{k}, k=1,2,\dots,K,
\end{equation} where $\Phi_{k}$ is the CDF of the variable $Z_{k}\sim E_{1}(0,1,\psi_{k})$. Combining \eqref{39},\eqref{42} and \eqref{43} finishes the proof.

\qed
\end{proof}

\begin{proposition}\label{4p}
Consider $\mathcal{F}_{0}$ defined in Assumption \ref{12} and $\mathcal{F}:=\mathcal{F}_{1}\times\cdots\times\mathcal{F}_{K}$ defined in Assumption \ref{ji9}. If $\tilde{F}_{k}$ is a Gaussian distribution $N(\mu_{k},\Sigma_{k})$, $k=0,1,\dots,K$, and $\Sigma_0$ is positive definite, then \eqref{Jobj} is equivalent to  
\begin{subequations}\label{okmb}
\begin{eqnarray}
& \min\limits_{\tau,y} & -\tau^{\top}\mu_{0}+\sqrt{2\delta_{0}\tau^{\top}\Sigma_{0}\tau},\\
& {\rm{s.t.}} & \tau^{\top}\mu_{k}+\Phi_{k}^{-1}(1-\tilde{y}_{k})\sqrt{\tau^{\top}\Sigma_{k}\tau}\ge\xi_{k}, k=1,2,\dots,K,\label{89u}\\
&& 0\le y_{k}\le 1, k=1,2,\dots,K,\label{zxcc}\\
&& \prod_{k=1}^{K}y_{k}\ge\hat{\epsilon},\\
&&\tau\in\Delta_{\beta,q},\label{zxcv}
\end{eqnarray}
\end{subequations} where $\tilde{y}_{k}=\inf\limits_{x\in(0,1)}\{\frac{e^{-\delta_{k}}x^{y_{k}}-1}{x-1}\}$ and $\Phi_{k}$ is the CDF of the standard Gaussian distribution $N(0,1)$.

\end{proposition}

\begin{proof}  If $\tilde{F}_{0}$ is a Gaussian distribution, by Table 1, $\psi_{0}(t)=e^{-t}$ and satisfies $\inf\limits_{t\le 0}e^{-t}= 1\ge e^{-\delta_{0}}$ for any radius $\delta_{0}$. Therefore we can use the conclusion of Theorem \ref{63}. The objective function in \eqref{3999} can be written as $-\tau^{\top}\mu_{0}+\frac{\tau^{\top}\Sigma_{0}\tau}{2\alpha}+\alpha\delta_{0}$, which reaches its minimum value at $\alpha=\sqrt{\frac{\tau^{\top}\Sigma_{0}\tau}{2\delta_{0}}}$ where $\Sigma_{0}$ is positive definite and $\tau^{\top}\Sigma_{0}\tau>0$. Taking the optimal value of $\alpha$, we have that the optimal value of \eqref{3999} is $-\tau^{\top}\mu_{0}+\sqrt{2\delta_{0}\tau^{\top}\Sigma_{0}\tau}$. Therefore \eqref{Jobj} is equivalent to \eqref{okmb}.
\end{proof}
\qed

Next we study the solution method of the optimization problem \eqref{okmb}. As $y_k$ and $\tau$ are both decision variables, \eqref{89u} is a nonconvex constraint and \eqref{okmb} is not convex. Moreover, $\tilde{y}_{k}$ here is a highly nonlinear function of $y_{k}$.
Thus, we propose a sequential convex approximation method to solve the nonconvex problem \eqref{okmb}. We decompose problem \eqref{okmb} into the following two subproblems where two subsets of variables are fixed alternatively. Firstly, we 
compute $\tilde{y}_{k}^n
=\inf\limits_{x\in(0,1)}\{\frac{e^{-\delta_{k}}x^{y_{k}^{n}}-1}{x-1}\}$, and update $\tau$ by solving
\begin{subequations}\label{fgtt}
\begin{eqnarray}
& \min\limits_{\tau} & -\tau^{\top}\mu_{0}+\sqrt{2\delta_{0}\tau^{\top}\Sigma_{0}\tau},\\
& {\rm{s.t.}} & \tau^{\top}\mu_{k}+\Phi_{k}^{-1}(1-\tilde{y}_{k}^{n})\sqrt{\tau^{\top}\Sigma_{k}\tau}\ge\xi_{k}, k=1,2,\dots,K,\label{377}\\
&& \tau\in\Delta_{\beta,q}.
\end{eqnarray}
\end{subequations}  Then we fix $\tau=\tau^n$ and update $y$ by solving
\begin{subequations}\label{tygg}
\begin{eqnarray}
& \min\limits_{y} & \sum_{k=1}^{K}\Gamma_{k}y_{k}\\
&{\rm s.t.} & \frac{1}{2}\le \tilde{y}_{k}\le 1-\Phi(\frac{\xi_{k}-{\tau^{n}}^{\top}\mu_{k}}{\sqrt{\tau^{n}\Sigma_{k}{\tau^{n}}^{\top}}}), k=1,2,...,K,\label{vghu}\\
&&  0\le y_{k}\le 1, k=1,2,...,K,\label{366}\\
&& \sum_{k=1}^{K}\log{y_{k}}\ge\log{\hat{\epsilon}},\label{367}
\end{eqnarray}
\end{subequations} where $\Gamma_{k}$ is a given searching direction and $\tilde{y}_{k}=\inf\limits_{x\in(0,1)}\{\frac{e^{-\delta_{k}}x^{y_{k}}-1}{x-1}\}$. We denote $\tilde{y}_{k}=\chi_{k}(y_k):=\inf\limits_{x\in(0,1)}\{\frac{e^{-\delta_{k}}x^{y_{k}}-1}{x-1}\}$. By the proof of Proposition 4 in \cite{jiang2016data}, the infimum of $\chi_{k}(y_k)$ is attained in the interval $(0,1)$. For any $0\le y_k\le 1$, $\chi_{k}(y_k)>0$. By the Envelope Theorem \cite{tercca2021envelope}, $\chi_{k}(y_k)$ is strictly monotonically decreasing w.r.t. $y_k$. Thus we can reformulate \eqref{vghu} as:
\begin{equation}\label{59}
\chi_{k}^{-1}\left(1-\Phi(\frac{\xi_{k}-{\tau^{n}}^{\top}\mu_{k}}{\sqrt{\tau^{n}\Sigma_{k}{\tau^{n}}^{\top}}})\right)\le y_k\le \chi_{k}^{-1}(\frac{1}{2}),
\end{equation}where $\chi^{-1}(\cdot)$ denotes the inverse function of $\chi(\cdot)$. And from the strict monotonicity of function $\chi$, both sides of \eqref{59} are unique at the interval $(0,1)$. We then design an algorithm based on the sequential convex approximation method, and line search method for $\chi^{-1}$, see Algorithm 1.


\begin{algorithm}[ht]
\SetAlgoLined
  \KwData{$\mu_k$, $\Sigma_k$, $\delta_k$, $\xi_k$, $\Delta_{\beta,q}$ $n_{max}$, $\hat{\epsilon}$, $\tilde{\epsilon}$, $\gamma$, $k=0,1,...,K$.}
  \KwResult{$\tau^{n}$, $V^{n}$.}
Set $n=0$\;
Choose an initial point $y^{0}=[y_{1}^{0},...,y_{K}^{0}]$ feasible for \eqref{366} and \eqref{367}\;
\While{$n\le n_{max}$ and $\Vert y^{n-1}-y^{n}\Vert\ge\tilde{\epsilon}$}{
Compute $\tilde{y}_{k}^n
=\inf\limits_{x\in(0,1)}\{\frac{e^{-\delta_{k}}x^{y_{k}^{n}}-1}{x-1}\}$. Solve problem \eqref{fgtt} with $\tilde{y}_{k}^n$. Let $\tau^{n}, V^{n}$ be an optimal solution and the optimal value of \eqref{fgtt} respectively. Let $\theta^{n}$ be the optimal dual multiplier vector 
to constraints \eqref{377} \;
Use the line search method to find $y^{Up\cdot n}_{k}=\chi_{k}^{-1}(\frac{1}{2})$ and $y^{Low\cdot n}_{k}=\chi_{k}^{-1}(1-\Phi(\frac{\xi_{k}{\tau^{n}}^{\top}\mu_{k}}{\sqrt{\tau^{n}\Sigma_{k}{\tau^{n}}^{\top}}}))$, $k=1,...,K$\;
Solve problem \eqref{tygg} 
where we replace \eqref{vghu} by $y_{k}^{Low\cdot n}\le y_k\le y_{k}^{Up\cdot n}$, $k=1,...,K$, and 
set $$\Gamma_{k}=\theta^{n}_{k}\cdot(\Phi^{-1})'(1-\tilde{y}_k^n)\sqrt{{\tau^n}^{\top}\Sigma_{k}\tau^n};$$
let $\tilde{y}$ be an optimal solution of problem \eqref{tygg}\;
$y^{n+1}\leftarrow y^{n}+\gamma(\tilde{y}-y^{n}), n\leftarrow n+1$. Here, $\gamma\in (0,1)$ is the step length.
}
\caption{A hybrid algorithm to solve problem \eqref{okmb}}
\end{algorithm}

Note that in practical numerical experiments, the function $(\Phi^{-1})'$ does not have a closed-form. We apply the following approximation $$\Phi^{-1}(x)\approx t-\frac{2.515517+0.802853\times t+0.010328\times t^2}{1+1.432788\times t+0.189269\times t^2+0.001308\times t^3}, t=\sqrt{-2\log{x}},$$ which holds an error bound of $4.5\times10^{-14}$ in (\cite{abramowitz1964handbook}, Page 933) to approximate $\Phi^{-1}$ here. 

\begin{remark}
    Algorithm 1 can be seen as a particular case of the alternate convex search or block-relaxation methods
\cite{gorski2007biconvex}. From Theorem 2 in \cite{liu2016stochastic}, we know that Algorithm 1 converges to a stationary point in a finite number of iterations and the returned value $V^{n}$ is an upper bound of problem \eqref{okmb}. When these sub-problems are all convex,
the objective function is continuous, the feasible set is closed, the alternate convex search algorithm converges monotonically to a partial optimal point (Theorem 4.7 \cite{gorski2007biconvex}). Furthermore, Algorithm 1 is exactly a hybrid one which combines the line search method to handle the non-linearity of function $\tilde{y}_{k}=\chi_{k}(y_k), k=1,2,...,K$. 
\end{remark}

\subsection{K-L J-DRCCMDP with elliptical mixture distribution}\label{se3.3}

In this Section, we assume the reference distribution in the marginal K-L ambiguity set is an elliptical mixture distribution and study the reformulation of the K-L J-DRCCMDP problem \eqref{Jobj}. As for the variable vector $r_{k}$, the PDF $f_{k}$ of $r_{k}$ is defined by $f_{k}(r_{k})=\sum\limits_{j=1}^{J_{k}}\omega_{j}^{k}f_{j}^{k}(r_{k})$, where $f_{j}^{k}(r_{k})$ is the density function which follows $E_{|\Lambda|}(\mu_{j}^{k},\Sigma_{j}^{k},\psi_{j}^{k})$ and $\sum\limits_{j=1}^{J_{k}}\omega_{j}^{k}=1$.

\begin{theorem}\label{7u8}
Consider $\mathcal{F}_{0}$ defined in Assumption \ref{12} and $\mathcal{F}$ defined in Assumption \ref{ji9}. Suppose the reference distribution $\tilde{F}_{k}=\sum\limits_{j=1}^{J_{k}}\omega_{j}^{k}\tilde{F}_{j}^{k}$ is an elliptical mixture distribution with density $\tilde{f}_{k}(r_{k})=\sum\limits_{j=1}^{J_{k}}\omega_{j}^{k}\tilde{f}_{j}^{k}(r_{k})$, where $\tilde{F}_{j}^{k}(r_k)$ and $\tilde{f}_{j}^{k}(r_{k})$ are the distribution function and density function of $E_{|\Lambda|}(\mu_{j}^{k},\Sigma_{j}^{k},\psi_{j}^{k})$, $j=1,...,J_{k}$, respectively, and $\sum\limits_{j=1}^{J_{k}}\omega_{j}^{k}=1$, $k=0,1,...,K$. We further assume that $\Sigma_{j}^{0}$ is positive definite, $\mu_{j}^{0}\le 0$, $\psi_{j}^{0}$ is continuous and 
$\inf\limits_{t\le 0}\left\{\min\limits_{1\le j\le J_0}\psi_{j}^{0}(t)\right\}\ge e^{-\delta_{0}}$, $j=1,...,J_0$. Then (K-L J-DRCCMDP) problem \eqref{Jobj} is equivalent to
\begin{subequations}
\begin{align}
 \min\limits_{\tau, \alpha, x, y, l, \hat{y}}&  \alpha\log{\left[\sum_{j=1}^{J_{0}}\omega_{j}^{0}\exp{(-\frac{\tau^{\top}\mu_{j}^{0}}{\alpha})}\psi_{j}^{0}(-\frac{\tau^{\top}\Sigma_{j}^{0}\tau}{2\alpha^2})\right]}+\alpha\delta_{0},\label{399}\\
{\rm{s.t.}} &\ \ \tau^{\top}\mu_{j}^{k}+(\Phi_{j}^{k})^{-1}(1-l_{j}^{k})\sqrt{\tau^{\top}\Sigma_{j}^{k}\tau}\ge\xi_{k},j=1,2,\dots,J_{k}, k=1,2,\dots,K, \\
& \sum_{j=1}^{J_{k}}\omega_{j}^{k}l_{j}^{k}\ge \hat{y}_{k}, k=1,2,...,K,\\
& \hat{y}_{k}\ge\frac{e^{-\delta_{k}}x_{k}^{y_{k}}-1}{x_{k}-1}, k=1,2,...,K,\\
& 0<x_{k}<1, 0\le y_{k}\le 1, 0\le \hat{y}_{k}\le 1, k=1,2,\dots,K,\\
& 0\le l_{j}^{k}\le 1, j=1,2,\dots,J_{k}, k=1,2,\dots,K,\\
& \sum_{k=1}^{K}y_{k}\ge \hat{\epsilon}, \alpha\ge 0,\\
& \tau\in\Delta_{\beta,q},
\end{align}
\end{subequations} where $\Phi_{j}^{k}$ is the cdf of the variable $Z_{j}^{k}\sim E_{1}(0,1,\psi_{j}^{k})$.
\end{theorem}

{\it Proof}  By Lemma \ref{98}, the objective function in \eqref{12a} is equivalent to
\begin{equation}\label{400}
\inf\limits_{\alpha\in[0,+\infty)} \alpha {\rm{log}} \mathbb{E}_{\tilde{F}_{0}}\left[\exp(-\frac{\tau^{\top}r_{0}}{\alpha})\right]+\alpha \delta_{0}.
\end{equation} 

As $\tilde{F}_0$ is an elliptical mixture distribution, we have that
\begin{subequations}\label{90}
\begin{align}
 \mathbb{E}_{\tilde{F}_0}\left[\exp{(-\frac{\tau^{\top}r_{0}}{\alpha})}\right] & =\int_{\Omega_0}\exp(-\frac{\tau^{\top}r_{0}}{\alpha})d\tilde{F}_{0}(r_{0})\\
& =\int_{\Omega_0}\exp(-\frac{\tau^{\top}r_{0}}{\alpha})d\left(\sum\limits_{j=1}^{J_0}\omega_{j}^{0}\tilde{F}_{j}^{0}(r_0)\right)\\
& =\sum\limits_{j=1}^{J_0}\omega_{j}^{0}\int_{\Omega_0}\exp(-\frac{\tau^{\top}r_{0}}{\alpha})\tilde{f}_{j}^{0}(r_0)dr_0\\
& =\sum_{j}\omega_{j}^{0}\exp{(-\frac{\tau^{\top}\mu_{j}^{0}}{\alpha})}\psi_{j}^{0}(-\frac{\tau^{\top}\Sigma_{j}^{0}\tau}{2\alpha^2}),
\end{align}
\end{subequations} where the last equation follows from Lemma \ref{3}. Then by the assumption in Theorem \ref{7u8}, we have
\begin{align*}
\alpha {\rm{log}} \mathbb{E}_{F_{0}}\left[\exp(-\frac{\tau^{\top}r_{0}}{\alpha})\right]+\alpha \delta_{0} & = \alpha\left[\sum_{j}\omega_{j}^{0}\exp{(-\frac{\tau^{\top}\mu_{j}^{0}}{\alpha})}\psi_{j}^{0}(-\frac{\tau^{\top}\Sigma_{j}^{0}\tau}{2\alpha^2})+\delta_0\right] \\
& \geqslant \alpha\left[\sum_{j}\omega_{j}^{0}\left\{\min\limits_{1\le j\le J_0}\psi_{j}^{0}(-\frac{\tau^{\top}\Sigma_{j}^{0}\tau}{2\alpha^2})\right\}+\delta_0\right] \\
& \geqslant \alpha\left[\sum_{j}\omega_{j}^{0}\inf\limits_{t\le 0}\left\{\min\limits_{1\le j\le J_0}\psi_{j}^{0}(t)\right\}+\delta_0\right]\ge 0.
\end{align*}
Here the first inequality is by the assumption $\mu_{j}^{0}\le 0$ and thus $\exp{(-\frac{\tau^{\top}\mu_{j}^{0}}{\alpha})}\ge 1$. The second inequality holds by the positive definiteness of $\Sigma_{j}^{0}$, 
and the fact that $\min\limits_{1\le j\le J_0}\psi_{j}^{0}(t)\ge e^{-\delta_{0}}\ge 0$, $\forall t\le 0$.
The last equality is by $\inf\limits_{t\le 0}\left\{\min\limits_{1\le j\le J_0}\psi_{j}^{0}(t)\right\}\ge e^{-\delta_{0}}$. As $\psi_{j}^{0}$ is continuous, by Weierstrass Theorem we have $$\inf\limits_{\alpha\in[0,+\infty)} \alpha {\rm{log}} \mathbb{E}_{F_{0}}\left[\exp(-\frac{\tau^{\top}r_{0}}{\alpha})\right]+\alpha \delta_{0}
=\min\limits_{\alpha\in[0,+\infty]} \alpha {\rm{log}} \mathbb{E}_{F_{0}}\left[\exp(-\frac{\tau^{\top}r_{0}}{\alpha})\right]+\alpha \delta_{0}.$$ Taking \eqref{90} to \eqref{400}, we get the reformulation in \eqref{399}.

Same as the proof of Theorem \ref{63}, we can first reformulate \eqref{JF} by introducing auxiliary variables $y_{k}\in\mathbb{R}_{+}$ as \eqref{41} and \eqref{42}. By Proposition 2 in \cite{hu2022chance}, we have that $$\mathbb{P}_{\tilde{F}_{k}}(\tau^{\top}\cdot r_k\ge \xi_{k})\ge \tilde{y}_{k}\iff \sum_{j=1}^{J_{k}}\omega_{j}^{k}\mathbb{P}_{\tilde{F}_{j}^{k}}(\tau^{\top}r_k\ge \xi_{k})\ge \tilde{y}_{k}, k=1,2,\dots,K,$$
where $\tilde{y}_{k}=\inf\limits_{x\in(0,1)}\{\frac{e^{-\delta_{k}}x^{y_{k}}-1}{x-1}\}$. With Theorem \ref{63}, through adding auxiliary variables $l_{j}^{k}, \hat{y}_{k}\in\mathbb{R}_{+}$ and $x_{k}\in (0,1)$, we can reformulate \eqref{41} as
\begin{equation}\label{433}
\mathbb{P}_{\tilde{F}_{j}^{k}}(\tau^{\top}r_k\ge \xi_{k})\ge l_{j}^{k}, j=1,2,...,J_{k},  k=1,\dots,K, 
\end{equation}
\begin{equation}\label{444}
\sum_{j=1}^{J_{k}}\omega_{j}^{k}l_{j}^{k}\ge \hat{y}_{k},  
 \hat{y}_{k}\ge\frac{e^{-\delta_{k}}x_{k}^{y_{k}}-1}{x_{k}-1}, k=1,...,K,
\end{equation}
\begin{equation}\label{51}
 0<x_{k}<1, \sum_{k=1}^{K}y_{k}\ge \hat{\epsilon}, 0\le y_{k}\le 1, 0\le l_{j}^{k}\le 1, j=1,2,\dots,J_{k}, k=1,2,\dots,K.
\end{equation}
As $\tilde{F}_{j}^{k}$ is an elliptical distribution, \eqref{433} is equivalent to
\begin{equation}\label{52}
\tau^{\top}\mu_{j}^{k}+(\Phi_{j}^{k})^{-1}(1-l_{j}^{k})\sqrt{\tau^{\top}\Sigma_{j}^{k}\tau}\ge\xi_{k}, j=1,2,\dots,J_{k}, k=1,2,\dots,K,
\end{equation} where $\Phi_{j}^{k}$ is the cdf of $E_{1}(0,1,\psi_{j}^{k})$. Collecting the reformulations \eqref{400}, \eqref{90}, \eqref{444}, \eqref{51} and \eqref{52}, we finish the proof.




\qed

\begin{remark}
    Any mixture distribution is a weighted sum of a finite set of probability measures, and can be seen as a semiparametric approach to model the randomness \cite{mclachlan1988mixture}. As a broader class of distributions, mixture distributions has now been used in many areas such as finance, economics and engineering. In particular, the elliptical mixture distribution is an important type for research. Theorem \ref{7u8} gives an exact tractable reformulation for K-L J-DRCCMDP problems when the reference distribution belongs to the class of elliptical mixture distribution under a mild assumption. 
\end{remark}

\section{Numerical experiments}\label{j4}

\subsection{Machine replacement problem}\label{4.1}
We carry out the numerical tests on a machine replacement problem \cite{delage2010percentile,goyal2022robust,ramani2022robust,varagapriya2022constrained,wiesemann2013robust}. In the machine replacement problem, we consider the opportunity cost in the objective function along with two kinds of opportunity cost in the constraint. The opportunity cost denoted by $r_0$ comes from the potential production losses when the machine is under repair. The maintenance cost comes from two parts: one part is due to the operation consumption for machines, such as the required electricity fees and fuel costs when the machine is working denoted by $r_1$; the other part comes from the production of low quality products denoted by $r_2$. 
We set the states as the using age of the machine. At each state,  there are two possible actions: $a_1$, repair and $a_2$, do not repair. The considered costs are incurred at every state. The transition probabilities are known for the whole MDP, and are the same as in \cite{varagapriya2022constrained}.

In all numerical experiments, we take the discount factor $\beta=0.9$ and assume that the initial distribution $q$ is a uniform distribution. We assume that there are 10 states. The mean values of the three costs are shown in Table 2. For example, at state 1, if the “repair” action is taken, the mean values of $r_0, r_1, r_2$ are $-10, -15, 0$, respectively; if the action “do not repair” is used, the mean values of three costs are $0, -10, -40$ respectively. The last two states are risky states such that the mean values of costs are much lower. The covariance matrices of the three costs are all assumed to be diagonal and positive definite. Concretely, for both actions, the covariance matrix of $r_0$ is $\Sigma_0=diag([0.3,0.3,0.3,0.3,0.3,0.3,0.3,0.3,0.3,0.3,0.3,0.3,0.3,0.3,0.3,3,5,2,8,9])$,\\ the covariance matrix of $r_1$ is $\Sigma_1=diag([0.5,5,0.5,0.5,0.5,5,0.5,5,0.5,0.5,0.5,\\5,0.5,0.5,0.5,0.5,8,9,8,9])$ and the the covariance matrix of $r_2$ is $\Sigma_2=diag([0.04,\\ 0.04,0.04,0.04,0.04,0.04,0.04,0.04,0.04,0.04,0.04,0.04, 0.04,0.04, 0.04,4,9,8,\\ 8.5,10])$.

In our numerical experiments, we set $\xi_1=\xi_2=-40$ and $\epsilon_1=\epsilon_2=0.8$ for I-DRCCMDP in \eqref{obj MDP}. We set $\xi_1=\xi_2=-40$ and $\hat{\epsilon}=0.8$ for J-DRCCMDP in \eqref{Jobj}. 
\begin{table}[htp]\label{yu}
\caption{\normalsize The mean values of three kinds of costs}
\centering
\begin{tabular}{c|cc|cc|cc}
\hline
\multirow{2}{*}{States} & \multicolumn{2}{c|}{Opportunity cost}     & \multicolumn{2}{c|}{Operation consumption cost}     & \multicolumn{2}{c}{Low quality cost}      \\ \cline{2-7} 
                   & \multicolumn{1}{c|}{$r_0(s,a_1)$} & $r_0(s,a_2)$ & \multicolumn{1}{c|}{$r_1(s,a_1)$} & $r_1(s,a_2)$ & \multicolumn{1}{c|}{$r_2(s,a_1)$} & $r_2(s,a_2)$ \\ \hline
1                  & \multicolumn{1}{c|}{-10} & 0 & \multicolumn{1}{c|}{-15} & -10 & \multicolumn{1}{c|}{0} & -40 \\ \hline
2                  & \multicolumn{1}{c|}{-10} & 0 & \multicolumn{1}{c|}{-15} & -30 & \multicolumn{1}{c|}{0} & -40 \\ \hline
3                  & \multicolumn{1}{c|}{-10} & 0 & \multicolumn{1}{c|}{-15} & -40 & \multicolumn{1}{c|}{0} & -50 \\ \hline
4                  & \multicolumn{1}{c|}{-10} & 0 & \multicolumn{1}{c|}{-15} & -50 & \multicolumn{1}{c|}{0} & -50 \\ \hline
5                  & \multicolumn{1}{c|}{-10} & 0 & \multicolumn{1}{c|}{-15} & -70 & \multicolumn{1}{c|}{-15} & -50 \\ \hline
6                  & \multicolumn{1}{c|}{-10} & 0 & \multicolumn{1}{c|}{-15} & -80 & \multicolumn{1}{c|}{-15} & -55 \\ \hline
7                  & \multicolumn{1}{c|}{-10} & 0 & \multicolumn{1}{c|}{-15} & -80 & \multicolumn{1}{c|}{-15} & -55 \\ \hline
8                  & \multicolumn{1}{c|}{-10} & 0 & \multicolumn{1}{c|}{-15} & -80 & \multicolumn{1}{c|}{-15} & -55 \\ \hline
9                  & \multicolumn{1}{c|}{-40} & -85 & \multicolumn{1}{c|}{-50} & -200 & \multicolumn{1}{c|}{-30} & -80 \\ \hline
10                  & \multicolumn{1}{c|}{-40} & -95 & \multicolumn{1}{c|}{-50} & -200 & \multicolumn{1}{c|}{-30} & -100 \\ \hline
\end{tabular}
\end{table}

\subsection{Numerical results on K-L divergence based DRCCMDP}
We consider the case where the reference distribution of the K-L divergence based ambiguity set is a Gaussian distribution with mean values and covariance matrix defined in Section \ref{4.1}. Following the same proof process in Proposition \ref{4p}, when the reference distribution is a Gaussian distribution, we can further reformulate the I-DRCCMDP problem from \eqref{ujn} as
\begin{subequations}\label{oguji}
\begin{eqnarray}
& \min\limits_{\tau} & -\tau^{\top}\mu_{0}+\sqrt{2\delta_{0}\tau^{\top}\Sigma_{0}\tau},\\
& {\rm{s.t.}} & \tau^{\top}\mu_{k}+\Phi_{k}^{-1}(1-\tilde{\epsilon}_{k})\sqrt{\tau^{\top}\Sigma_{k}\tau}\ge\xi_{k}, k=1,2,\dots,K,\\
&& \tau\in \Delta_{\beta,q},
\end{eqnarray}
\end{subequations} where $\tilde{\epsilon}_k=\inf\limits_{x\in(0,1)}\{\frac{e^{-\delta_{k}}x^{\epsilon_{k}}-1}{x-1}\}$. We consider six different cases when $\delta_{0}=\delta_1=\delta_2=0.5, 0.4, 0.3, 0.2, 0.1, 0.01$, respectively. We solve the convex optimization problem \eqref{oguji} using GUROBI in MATLAB, on a computer with AMD Ryzen 7 5800H CPU and 16.0 GB RAM. We show the probability of the “repair” action at each state under K-L I-DRCCMDP in Figure (a). As there are in total two actions to choose at each state, the probability of the “do not repair” action can be computed by subtracting the probability of the “repair” action with 1. From Figure (a), we can see the asymptotic convergence of the probability at each state when $\delta_0, \delta_1, \delta_2$ decrease from $0.5$ to $0.01$. Under all six radii, the probability of “repair” at last three states are all equal to $1$, which is consistent to the fact that the machine gets aging as the state goes by.

\begin{figure} \centering
\subfigure[K-L I-DRCCMDP] { \label{aesd}    
\includegraphics[width=0.487\columnwidth]{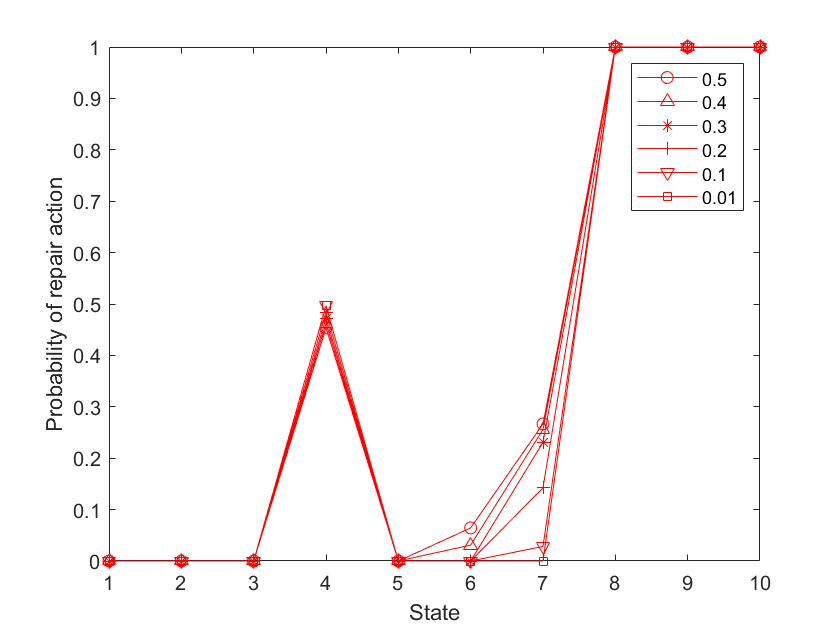}}     
\subfigure[K-L J-DRCCMDP] { \label{aesd2} 
\includegraphics[width=0.487\columnwidth]{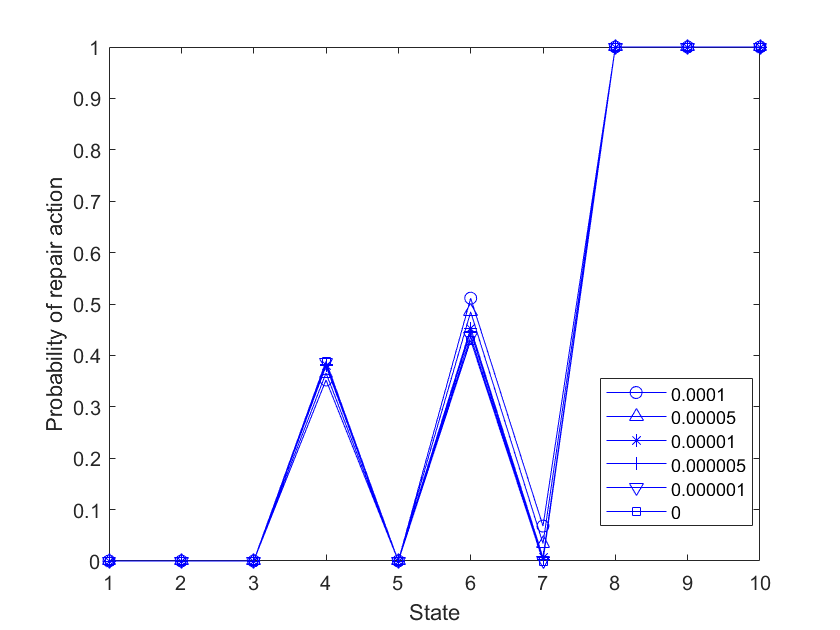}} 
\caption{The probability of the “repair” action at each state} 
\end{figure}

Next we focus on the K-L J-DRCCMDP problem which is solved by Algorithm 1. We set the initial points $y_1^0=0.95, y_2^0=0.91$ and $n_{max}=50$, $\tilde{\epsilon}=10^{-4}$, $\gamma=0.9$, the approximation accuracy of the line search method is $10^{-3}$. The other settings and parameters are the same as K-L I-DRCCMDP. We consider six cases with different radii, $\delta_0=\delta_1=\delta_2=10^{-4}, 5\times10^{-5}, 10^{-5}, 5\times10^{-6}, 10^{-6}, 0$. We use the MOSEK solver to solve the sub problem \eqref{fgtt} and \eqref{tygg}. We list the probability of “repair” at each state in Figure (b), from which we can observe the 
 convergence of the probability at each state when the radius decreases to $0$.


\section{Conclusion}
In this paper, we study the 
distributionally robust chance constrained Markov decision process  
problems. We consider Kullback-Leibler distance based ambiguity sets centered at 
elliptical distributions 
or elliptical mixture distributions. 
We find the deterministic formulation of the K-L I-DRCCMDP problem 
and design 
a hybrid algorithm to solve the K-L J-DRCCMDP problem. 
However, we only consider the randomness of the reward vector with deterministic transition probability.
It is a promising topic to consider the randomness of the transition probability in this kind of problems.
Moreover, we can study the reformulations and efficient algorithms under more ambiguity sets.

\section*{Declarations}
\begin{itemize}
\item 
{\bf Funding:} This research was supported by National Key R\&D Program of China under No. 2022YFA1004000 and National Natural Science Foundation of China under Grant Number 11991023 and 11901449.
    \item 
{\bf Competing interests:}
The authors declare they have no financial interests.
    \item 
{\bf Data Availability:} All data generated or analysed during this study are included in this published article

\end{itemize}

\bibliographystyle{plain}

\bibliography{Referencece} 

\begin{thebibliography}{10}

\bibitem{abramowitz1964handbook}
Milton Abramowitz and Irene~A Stegun.
\newblock {\em Handbook of Mathematical Functions with Formulas, Graphs, and
  Mathematical Tables}, volume~55.
\newblock US Government printing office, 1964.

\bibitem{altman1999constrained}
Eitan Altman.
\newblock {\em Constrained Markov Decision Processes: Stochastic Modeling}.
\newblock Routledge, 1999.

\bibitem{chakraborty2019capturing}
Souradeep Chakraborty.
\newblock Capturing financial markets to apply deep reinforcement learning.
\newblock {\em arXiv preprint arXiv:1907.04373}, 2019.

\bibitem{chen2022data}
Zhi Chen, Daniel Kuhn, and Wolfram Wiesemann.
\newblock Data-driven chance constrained programs over wasserstein balls.
\newblock {\em Operations Research}, 2022.

\bibitem{delage2010percentile}
Erick Delage and Shie Mannor.
\newblock Percentile optimization for markov decision processes with parameter
  uncertainty.
\newblock {\em Operations Research}, 58(1):203--213, 2010.

\bibitem{delage2010distributionally}
Erick Delage and Yinyu Ye.
\newblock Distributionally robust optimization under moment uncertainty with
  application to data-driven problems.
\newblock {\em Operations Research}, 58(3):595--612, 2010.

\bibitem{dvorkin2019chance}
Yury Dvorkin.
\newblock A chance-constrained stochastic electricity market.
\newblock {\em IEEE Transactions on Power Systems}, 35(4):2993--3003, 2019.

\bibitem{fang2018symmetric}
Kai~Wang Fang.
\newblock {\em Symmetric multivariate and related distributions}.
\newblock Chapman and Hall/CRC, 2018.

\bibitem{gao2022distributionally}
Rui Gao and Anton Kleywegt.
\newblock Distributionally robust stochastic optimization with wasserstein
  distance.
\newblock {\em Mathematics of Operations Research}, 2022.

\bibitem{gorski2007biconvex}
Jochen Gorski, Frank Pfeuffer, and Kathrin Klamroth.
\newblock Biconvex sets and optimization with biconvex functions: a survey and
  extensions.
\newblock {\em Mathematical methods of operations research}, 66(3):373--407,
  2007.

\bibitem{goyal2022robust}
Vineet Goyal and Julien Grand-Clement.
\newblock Robust markov decision processes: Beyond rectangularity.
\newblock {\em Mathematics of Operations Research}, 2022.

\bibitem{hamada2008capm}
Mahmoud Hamada and Emiliano~A Valdez.
\newblock Capm and option pricing with elliptically contoured distributions.
\newblock {\em Journal of Risk and Insurance}, 75(2):387--409, 2008.

\bibitem{hanasusanto2015distributionally}
Grani~A Hanasusanto, Vladimir Roitch, Daniel Kuhn, and Wolfram Wiesemann.
\newblock A distributionally robust perspective on uncertainty quantification
  and chance constrained programming.
\newblock {\em Mathematical Programming}, 151(1):35--62, 2015.

\bibitem{hu2013kullback}
Zhaolin Hu and L~Jeff Hong.
\newblock Kullback-leibler divergence constrained distributionally robust
  optimization.
\newblock {\em Available at Optimization Online}, pages 1695--1724, 2013.

\bibitem{hu2022chance}
Zhaolin Hu, Wenjie Sun, and Shushang Zhu.
\newblock Chance constrained programs with gaussian mixture models.
\newblock {\em IISE Transactions}, 54(12):1117--1130, 2022.

\bibitem{ji2021data}
Ran Ji and Miguel~A Lejeune.
\newblock Data-driven distributionally robust chance-constrained optimization
  with wasserstein metric.
\newblock {\em Journal of Global Optimization}, 79(4):779--811, 2021.

\bibitem{jiang2016data}
Ruiwei Jiang and Yongpei Guan.
\newblock Data-driven chance constrained stochastic program.
\newblock {\em Mathematical Programming}, 158(1):291--327, 2016.

\bibitem{jiang2022data}
Shiyi Jiang, Jianqiang Cheng, Kai Pan, Feng Qiu, and Boshi Yang.
\newblock Data-driven chance-constrained planning for distributed generation: A
  partial sampling approach.
\newblock {\em IEEE Transactions on Power Systems}, 2022.

\bibitem{joyce2011kullback}
James~M Joyce.
\newblock Kullback-leibler divergence.
\newblock In {\em International Encyclopedia of Statistical Science}, pages
  720--722. Springer, 2011.

\bibitem{kiran2021deep}
B~Ravi Kiran, Ibrahim Sobh, Victor Talpaert, Patrick Mannion, Ahmad~A
  Al~Sallab, Senthil Yogamani, and Patrick P{\'e}rez.
\newblock Deep reinforcement learning for autonomous driving: A survey.
\newblock {\em IEEE Transactions on Intelligent Transportation Systems}, 2021.

\bibitem{klabjan2013robust}
Diego Klabjan, David Simchi-Levi, and Miao Song.
\newblock Robust stochastic lot-sizing by means of histograms.
\newblock {\em Production and Operations Management}, 22(3):691--710, 2013.

\bibitem{kuccukyavuz2022chance}
Simge K{\"u}{\c{c}}{\"u}kyavuz and Ruiwei Jiang.
\newblock Chance-constrained optimization under limited distributional
  information: a review of reformulations based on sampling and distributional
  robustness.
\newblock {\em EURO Journal on Computational Optimization}, page 100030, 2022.

\bibitem{liu2016stochastic}
Jia Liu, Abdel Lisser, and Zhiping Chen.
\newblock Stochastic geometric optimization with joint probabilistic
  constraints.
\newblock {\em Operations Research Letters}, 44(5):687--691, 2016.

\bibitem{liu2022distributionally}
Jia Liu, Abdel Lisser, and Zhiping Chen.
\newblock Distributionally robust chance constrained geometric optimization.
\newblock {\em Mathematics of Operations Research}, 47(4):2950--2988, 2022.

\bibitem{ma2019state}
Shuai Ma and Jia~Yuan Yu.
\newblock State-augmentation transformations for risk-sensitive reinforcement
  learning.
\newblock In {\em Proceedings of the AAAI Conference on Artificial
  Intelligence}, volume~33, pages 4512--4519, 2019.

\bibitem{mannor2016robust}
Shie Mannor, Ofir Mebel, and Huan Xu.
\newblock Robust mdps with k-rectangular uncertainty.
\newblock {\em Mathematics of Operations Research}, 41(4):1484--1509, 2016.

\bibitem{mclachlan1988mixture}
Geoffrey~J McLachlan and Kaye~E Basford.
\newblock {\em Mixture models: Inference and applications to clustering},
  volume~38.
\newblock M. Dekker New York, 1988.

\bibitem{mcneil2015quantitative}
Alexander~J McNeil, R{\"u}diger Frey, and Paul Embrechts.
\newblock {\em Quantitative risk management: concepts, techniques and
  tools-revised edition}.
\newblock Princeton university press, 2015.

\bibitem{nguyen2022distributionally}
Hoang~Nam Nguyen, Abdel Lisser, and Vikas~Vikram Singh.
\newblock Distributionally robust chance-constrained markov decision processes.
\newblock {\em arXiv preprint arXiv:2212.08126}, 2022.

\bibitem{peng2021games}
Shen Peng, Abdel Lisser, Vikas~Vikram Singh, Nalin Gupta, and Eshan
  Balachandar.
\newblock Games with distributionally robust joint chance constraints.
\newblock {\em Optimization Letters}, 15(6):1931--1953, 2021.

\bibitem{prashanth2014policy}
LA~Prashanth.
\newblock Policy gradients for cvar-constrained mdps.
\newblock In {\em International Conference on Algorithmic Learning Theory},
  pages 155--169. Springer, 2014.

\bibitem{ramani2022robust}
Sivaramakrishnan Ramani and Archis Ghate.
\newblock Robust markov decision processes with data-driven, distance-based
  ambiguity sets.
\newblock {\em SIAM Journal on Optimization}, 32(2):989--1017, 2022.

\bibitem{satia1973markovian}
Jay~K Satia and Roy~E Lave~Jr.
\newblock Markovian decision processes with uncertain transition probabilities.
\newblock {\em Operations Research}, 21(3):728--740, 1973.

\bibitem{sutton1999policy}
Richard~S Sutton, David McAllester, Satinder Singh, and Yishay Mansour.
\newblock Policy gradient methods for reinforcement learning with function
  approximation.
\newblock {\em Advances in neural information processing systems}, 12, 1999.

\bibitem{sutton1999between}
Richard~S Sutton, Doina Precup, and Satinder Singh.
\newblock Between mdps and semi-mdps: A framework for temporal abstraction in
  reinforcement learning.
\newblock {\em Artificial intelligence}, 112(1-2):181--211, 1999.

\bibitem{tercca2021envelope}
Gon{\c{c}}alo Ter{\c{c}}a and David Wozabal.
\newblock Envelope theorems for multistage linear stochastic optimization.
\newblock {\em Operations Research}, 69(5):1608--1629, 2021.

\bibitem{varagapriya2022constrained}
V~Varagapriya, Vikas~Vikram Singh, and Abdel Lisser.
\newblock Constrained markov decision processes with uncertain costs.
\newblock {\em Operations Research Letters}, 50(2):218--223, 2022.

\bibitem{varagapriya2022joint}
V~Varagapriya, Vikas~Vikram Singh, and Abdel Lisser.
\newblock Joint chance-constrained markov decision processes.
\newblock {\em Annals of Operations Research}, pages 1--23, 2022.

\bibitem{wang2022reliable}
Jie Wang, Rui Gao, and Hongyuan Zha.
\newblock Reliable off-policy evaluation for reinforcement learning.
\newblock {\em Operations Research}, 2022.

\bibitem{wei2011point}
Junqing Wei, John~M Dolan, Jarrod~M Snider, and Bakhtiar Litkouhi.
\newblock A point-based mdp for robust single-lane autonomous driving behavior
  under uncertainties.
\newblock In {\em 2011 IEEE International Conference on Robotics and
  Automation}, pages 2586--2592. IEEE, 2011.

\bibitem{wiesemann2013robust}
Wolfram Wiesemann, Daniel Kuhn, and Ber{\c{c}} Rustem.
\newblock Robust markov decision processes.
\newblock {\em Mathematics of Operations Research}, 38(1):153--183, 2013.

\bibitem{wiesemann2014distributionally}
Wolfram Wiesemann, Daniel Kuhn, and Melvyn Sim.
\newblock Distributionally robust convex optimization.
\newblock {\em Operations Research}, 62(6):1358--1376, 2014.

\bibitem{xia2020risk}
Li~Xia.
\newblock Risk-sensitive markov decision processes with combined metrics of
  mean and variance.
\newblock {\em Production and Operations Management}, 29(12):2808--2827, 2020.

\bibitem{xie2021distributionally}
Weijun Xie.
\newblock On distributionally robust chance constrained programs with
  wasserstein distance.
\newblock {\em Mathematical Programming}, 186(1):115--155, 2021.

\bibitem{yu2022zero}
Zhihui Yu, Xianping Guo, and Li~Xia.
\newblock Zero-sum semi-markov games with state-action-dependent discount
  factors.
\newblock {\em Discrete Event Dynamic Systems}, 32(4):545--571, 2022.

\end{thebibliography}

\end{document}